\documentclass[12pt]{article}
\usepackage{fullpage}
\usepackage{amssymb}
\usepackage{amsthm}
\usepackage{graphicx}
\usepackage[all,2cell]{xy}
\newtheorem{theorem}{Theorem}[section]

\newtheorem{lemma}[theorem]{Lemma}

\newtheorem{proposition}[theorem]{Proposition}

\begin{document}
\title{Torus quotients of homogeneous spaces- minimal dimensional Schubert Varieties admitting semi-stable points.}
\author{S.S.Kannan, S.K.Pattanayak  \\  
\\ Chennai Mathematical Institute, Plot H1, SIPCOT IT Park,\\ Padur 
Post Office, Siruseri, Tamilnadu - 603103, India.\\
kannan@cmi.ac.in, santosh@cmi.ac.in} 

\maketitle
\date{30.07.2008}

\begin{abstract} In this paper, for any simple, simply connected algebraic group $G$ of type $B_n,C_n$ or $D_n$ and for any maximal parabolic subgroup $P$ of $G$, we describe all minimal dimensional Schubert varieties
in $G/P$ admitting semistable points for the action of a maximal torus $T$ with respect to an ample 
line bundle on $G/P$. In this paper, we also describe, for any semi-simple simply connected algebraic group $G$ and for any Borel subgroup $B$ of $G$, all Coxeter elements $\tau$ for which the Schubert variety $X(\tau)$ admits
a semistable point for the action of the torus $T$ with respect to a non-trivial line bundle on
 $G/B$.

\end{abstract}
\hspace*{1cm}Keywords: Semistable points, line bundle, Coxeter element. 

\section {Introduction}

Let $G$ be a simply connected semi-simple algebraic group over an algebraic closed field $k$. 
Let $T$ be a maximal torus of $G$ and let $B$ be a Borel subgroup of $G$ containing $T$. In [4] and [5],  
the parabolic subgroups $Q$ of $G$ containing $B$ for which there exists an ample line bundle 
$\mathcal {L}$ on $G/Q$ such that the semistable points $(G/Q)_T^{ss}(\mathcal L)$ are the same as 
the stable points $(G/Q)_T^{s}(\mathcal L)$.

In [7], when $Q$ is a maximal parabolic subgroup of $G$ and $\mathcal L = \mathcal L_{\varpi}$, 
where $\varpi$ is a minuscule dominant weight, it is shown that there exists unique minimal 
dimensional Schubert variety $X(w)$ admitting semistable points with respect to $\mathcal L$.

Now, let $G$ be a simple algebraic group of type $B, C$ or $D$ and let $P$ be a maximal parabolic subgroup of $G$. Let $\mathcal L$ be an ample line bundle on $G/P$. In this paper, we describe all minimal dimensional Schubert varieties in $G/P$ admitting semistable points with respect to $\mathcal L$. 
For a precise statement, see theorem 3.2.

Now, let $G$ be a semi-simple simply connected algebraic group over an algebraic closed field $k$. Let $T$ be a maximal torus of $G$ and let $B$ be a Borel subgroup of $G$ containing $T$. A Schubert variety $X(w)$ in $G/B$ contains a (rank $G$)-dimensional $T$-orbit if and only if $w \geq \tau$ for some Coxeter element $\tau$.

So, it is a natural question to ask if for every Coxeter element $\tau$, there is a non-trivial line bundle $\mathcal L$ on $G/B$ such that $X(w)_T^{ss}(\mathcal L) \neq \emptyset$.

In this paper, we describe all Coxeter elements $\tau$ for which there exists a non-trivial 
line bundle $\mathcal L$ on $G/B$ such that $X(w)_T^{ss}(\mathcal L) \neq \emptyset$. 

The layout of the paper is as follows: 

Section 2 consists of preliminary notation and a combinatorial lemma.

Section 3 consists of minimal dimensional Schubert varieties in $G/P$, ( where $G$ is a 
semi-simple  algebraic group of type $B_n, C_n$ or $D_n$ and $P$ is a maximal parabolic 
subgroup of $G$), admitting semistable points with respect to an ample line bundle on $G/P$. 

Section 4 consists of description of Coxeter elements for which the corresponding Schubert 
varieties admit semistable points with respect to a non-trivial line bundle on $G/B$.

\section{Preliminary notation and a combinatorial lemma}
This section consists of preliminary notation and a lemma describing a criterion for a Schubert variety to admit semistable points.
Let $G$ be a semi-simple algebraic group over an algebraically closed field $k$. 
Let $T$ be a 
maximal torus of $G$, $B$ a Borel subgroup of 
$G$ containing $T$ and let $U$ be the unipotent radical of $B$. 
Let $N_G(T)$ be the normaliser of $T$ in $G$. Let $W=N_G(T)/T$ be Weyl group of 
$G$ with respect to $T$ and $R$ denote the set of roots with respect to 
$T$, $R^{+}$ positive roots with respect to $B$. Let $U_{\alpha}$ denote 
the one dimentional $T$-stable subgroup of $G$ corresponding to the root 
$\alpha$ and let $S=\{\alpha_1, \cdots, \alpha_l\}\subseteq R^{+}$ denote the set of simple roots.  
For a subset $I\subseteq S$ denote 
$W^I =\{w\in W| w(\alpha)>0, \, \alpha\in I\}$ and $W_I$ is the subgroup of $W$
generated by the simple reflections $s_\alpha, \alpha \in I$. Then every $w \in W$ can be uniquely 
expressed as $w=w^I.w_I$, with $w^I \in W^I$ and $w_I \in W_I$. Denote $R(w)=
\{\alpha \in R^+: w(\alpha)<0\}$ and $w_0$ is the longest element of $W$ with respect to $S$.
Let $X(T)$ (resp. $Y(T)$) denote the set of characters of $T$ (resp. 
one parameter subgroups of $T$). Let $E_1:= X(T)\otimes \mathbb R$, 
$E_2=Y(T)\otimes \mathbb R$. Let $\langle . ,  .\rangle : E_1\times E_2 
\longrightarrow \mathbb R$ be the canonical non-degenerate bilinear form. 
Choose $\lambda_j$'s in $E_2$ such that 
$\langle \alpha_i, \lambda_j \rangle = \delta_{ij}$ for all $i$. Let 
$\overline{C} := \{\lambda \in E_2 | \langle \lambda , \alpha \rangle \geq 0 \,\,\forall\,\alpha \in R^+\}$ and for all $\alpha \in R$, there is a homomorphism $SL_2 \stackrel{\phi_\alpha}{\longrightarrow} G$, see page-19 of [1]. We have $\check{\alpha} : G_m \longrightarrow G$ defined by $\check{\alpha}(t) = \phi_\alpha ( \left( \begin{array}{cc}
t & 0 \\
0 & t^{-1} \\
\end{array} \right)).$
We also have 
$s_{\alpha}(\chi)=\chi-\langle \chi, \check{\alpha}\rangle \alpha$ for all 
$\alpha \in R$ and $\chi \in E_1$. Set $s_i=s_{\alpha_i} \,\ \forall 
\,\ i=1,2, \cdots, l$. Let $\{\omega_i: i=1,2,\cdots, l\} \subset E_1$ be the 
fundamental weights; i.e. $\langle \omega_i, \check{\alpha_j} \rangle = 
\delta_{ij}$ for all $i,j = 1,2, \cdots, l$. 


For any character $\chi$ of $B$, we denote by $\mathcal L_{\chi}$, the line bundle on $G/B$ 
given by the character $\chi$. Let  $X(w) = \overline{BwB/B}$ denote the Schubert variety 
 corresponding to $w$. We denote by $X(w)_T^{ss}(\mathcal L_{\chi})$ the semistable points of 
$X(w)$ for the action of $T$ with respect to the line bundle $\mathcal L_{\chi}$.

\begin{lemma}

Let $\chi = \sum_{\alpha \in S}a_{\alpha}\varpi_{\alpha}$ be a dominant character of $T$ which is in the root lattice. Let $I = Supp (\chi) = \{\alpha \in S : a_{\alpha}\neq 0\}$ and let $w\in W^{I^c}$. Then 
$X(w)_{T}^{ss}(\mathcal {L}_{\chi}) \neq \emptyset $ if and only if $w\chi \leq 0$.
\end{lemma}

\begin{proof}
If $X(w)_{T}^{ss}(\mathcal {L}_{\chi}) \neq \emptyset $, then, by Hilbert-Mumford criterion 
(Theorem (2.1) of [8]) and lemma (2.1) of [10], we see that $w\chi \leq 0$.

Conversely, let $w\chi \leq 0$.

\underline {Step 1} - We prove that if $w, \tau \in W^{I^c}$ are such that $X(w)\subseteq \bigcup_{\phi \in W}\phi X(\tau)$, then, $w \leq \tau$. Now, suppose that $X(w)\subseteq \bigcup_{\phi \in W}\phi X(\tau)$. Then, since $X(w)$ is irreducible 
and $W$ is finite, we must have

\hspace{3cm} $ X(w) \subseteq \phi X(\tau),\,\, \mbox{for some} \,\ \phi \in W.$
 
 Hence, $\phi^{-1}X(w) \subseteq X(\tau)$. Now, let $P_I= BW_IB$ and consider 
the projection 
\begin{center}
$\pi: G/B \longrightarrow G/P_I$
\end{center}

Then, $\pi^{-1}(\phi^{-1}X(w)) \subseteq \pi^{-1}(X(\tau))$. Let $w^{max}$(resp. 
$\tau^{max}$) be the maximal representative of $w$ (resp. $\tau$) in $W$.
 
Hence, $\phi^{-1}X(w^{max}) \subseteq X(\tau^{max})$. So, we may assume that $I = S$.

Now, since $\phi^{-1}X(w) \subseteq X(\tau)$, we have  $\phi^{-1}w_1 \leq 
\tau, \,\, \forall \,\, w_1 \leq w$.

Therefore $w_1\phi \leq \tau^{-1}\,\, \forall \,\, w_1 \leq w^{-1}$. Hence, by lemma (5.6) of [6], we have $\tau^{-}(w^{-1},\phi^{-1})\phi \leq \tau^{-1}$.

Hence, $w^{-1}\leq \tau^{-}(w^{-1},\phi^{-1})\phi \leq \tau^{-1}$. So $w \leq 
\tau$.

Now, let $w \in W^{I^c}$ be such that $w\chi \leq 0$. Then by step 1, there exist a point $x\in X(w) \setminus W$-translates of $X(\tau), 
\tau \in W^{I^c},\,\, \tau \ngeq w$. $\hspace{4cm} \longrightarrow (1)$.

\underline{Step 2}: We prove that $x$ is semistable.

Let $\lambda$ be an one parameter subgroup of $T$. Choose $\phi \in W$ such that $\phi \lambda 
\in \bar C$. Let $\tau \in  W^{I^c}$ be such that $\phi x \in U_{\tau}\tau P_I$.

 By (1), \,\ $w \leq \tau$. Hence, $\tau \chi \leq w\chi \leq 0$.

 Hence, by lemma (2.1) of [10], we have  $\mu^L(x,\lambda)= \mu^L(\phi x, \phi \lambda)= 
\langle -\tau \chi, \lambda \rangle \geq 0$. 

Hence, by Hilbert-Mumford criterion (Theorem (2.1) of [8]), $x$ is semistable. 

\end{proof}

\section{Minimal dimensional Schubert variety in $G/P$ admitting semistable points}

In this section, we describe all minimal dimensional Schubert varieties $X(w)$
in $G/P$ (where $G$ is a simple algebraic group of type $B$, $C$ or $D$, and $P$ is a maximal parabolic subgroup of $G$) for which $X(w)$ admits a semistable point for the action of a maximal torus of $G$ with respect to an ample line bundle on $G/P$.

Let $I_r = S \setminus \{\alpha_r\}$ and let $P_{I_r}= BW_{I_r}B$ be the maximal parabolic corresponding to the simple root $\alpha_r$.
In this section we will describe all minimal elememts  of $ W^{I_r}$ for which 
$X(w)_{T}^{ss}(\mathcal L_r) \neq \emptyset$.

At this point, we recall a standard property of the fundamental weights of type $A_n, B_n, C_n$ and $D_n$.

In types $A_n, B_n, C_n$ and $D_n$, we have $|\langle \varpi_r, 
\check{\alpha}\rangle| \leq 2$ for any fundamental weight $\varpi_r$ and any root $\alpha$.

\begin{proof}
Now $\langle \varpi_r,\check{\alpha}\rangle \leq \langle \varpi_r,\check{\eta}\rangle$, where $\eta$ is a highest root for the corresponding root system.

The highest root for type $A_n$ is $\alpha_1+\alpha_2+\ldots +\alpha_n$, the highest roots for type $B_n$ are $\alpha_1+2(\alpha_2+ \ldots +\alpha_n)$ and $\alpha_1+\alpha_2+ \ldots +\alpha_n$, the highest roots for type $C_n$ are $2(\alpha_1+\alpha_2+ \ldots +\alpha_{n-1})+\alpha_n$ and $\alpha_1+2(\alpha_2+ \ldots +\alpha_{n-1})+\alpha_n$ and the unique highest root for type $D_n$ is  $\alpha_1+2(\alpha_2+ \ldots +\alpha_{n-2})+\alpha_{n-1}+\alpha_n$.

In all these cases, we have $\langle \varpi_r,\check{\eta}\rangle \leq 2$. So  $|\langle \varpi_r,\check{\alpha}\rangle| \leq 2$, for any root $\alpha$.
\end{proof}    

Let $G$ be a simple simply-connected algebraic group of type $B_n, C_n$ or $D_n$. Let $T$ be a maximal torus of $G$ and let $S$ be the set of simple roots with respect to a Borel subgroup $B$ of $G$ containing $T$.

\begin{proposition}
Let $I_r = S \setminus \{\alpha_r\}$ and let $w \in W^{I_r}$ be of maximal 
length such that $w(\varpi_r) \geq 0$. Write $w(\varpi_r) = \sum_{i=1}^{n}a_i\alpha_i$ and let $a = max \{a_i: i = 1,2, \cdots ,n\}$. Then $a \in \{1,\frac {3}{2}\}$. Further, if $a= \frac {3}{2}$, then $r$ must be odd and $\varpi_r$ must be in type $D_n$ and $a=a_{n-1}$ or $a= a_n$.    

\end{proposition}

\begin{proof}

Since $2\leq r \leq n-2$, we have $2\varpi_r \in Z_{\geq 0}S$. Hence, if $a \in \{1,\frac {3}{2}\}$, then $a \geq 2$.

 Let $i_{0}$ be the least integer such that $a_{i_0}=a$. 

Clearly, $i_0 \neq 1$. We first observe that, $s_{i_0}w(\varpi_r)= w(\varpi_r)- \langle 
w(\varpi_r), \check{\alpha_{i_0}} \rangle \alpha_{i_0} \geq 0$, since, $\langle 
w(\varpi_r), \check{\alpha_{i_0}} \rangle \leq 2 \leq a=a_{i_0}$.

For all the cases except $i_0=n$ in type $B_n$, $i_0=n-1$ in type $C_n$ and  $i_0=n-2, n-1, n$ in type $D_n$, we have $\langle 
w(\varpi_r), \check{\alpha_{i_0}} \rangle = 2a-(a_{i_0-1}+a_{i_0+1}) > 0$. Hence, 
$s_{\alpha_{i_0}}w(\varpi_r) < w(\varpi_r)$. So, $s_{\alpha_{i_0}}w > w$, a contradiction to the maximality  of $w$.

Now, we treat the special cases explicitly. 

{\it Case 1} : $i_0 = n$ in type $B_n$.

 In this case,  $\langle w(\varpi_r), \check{\alpha_n} \rangle = -2a_{n-1}+2a_n > 0$, 
since $a_n =a > a_{n-1}$. So, $s_nw(\varpi_r) < w(\varpi_r)$. Hence, $s_nw > w$, a contradiction to the maximality of $w$.

{\it Case 2} : $i_0 = n-1$ in type $C_n$.

In this case   $\langle w(\varpi_r), \check{\alpha_{n-1}} \rangle =-a_{n-2}+2a_{n-1}-2a_n$.

So, we need to show that $2a_{n-1} > a_{n-2}+2a_n$. If not, then $2a_n \geq a_{n-1}+1$, since $a_{n-2} \leq a_{n-1}-1$.

Now, we have $s_nw(\varpi_r)= \sum_{i \neq n}a_i\alpha_i + (a_{n-1}-a_n)\alpha_n \geq 0$, since $a_{n-1}=a \geq a_n$.

On the other hand, since $2a_n \geq a_{n-1}+1$, we have $a_{n-1}-a_n \leq a_n-1$. So, $s_nw(\varpi_r) < w(\varpi_r)$.  Hence, $s_nw > w$, a contradiction to the maximality of $w$.

{\it Case 3} : $i_0 = n$ in type $D$.

Here, we have  $\langle w(\varpi_r), \check{\alpha_n} \rangle = 2a_n-a_{n-2} > 0$, since 
$a_n=a > a_{n-2}$.

So, 
$s_nw(\varpi_r) < w(\varpi_r)$. Hence, $s_nw > w$, a contradiction to the maximality  of $w$.

{\it Case 4} : $i_0 = n-1$ in type $D$.

This case is similar to Case-3.

{\it Case 5} : $i_0 = n-2$ in type $D$.

We have $\langle w(\varpi_r), \check{\alpha_{n-2}} \rangle = -a_{n-3}+2a_{n-2}-a_{n-1}-a_n$.

In order to prove that $\langle w(\varpi_r), \check{\alpha_{n-2}} \rangle > 0$, we need to prove $a_{n-1}+a_n \leq a_{n-2}$, since $a_{n-3} < a_{n-2}$.

Suppose  $a_{n-1}+a_n \geq a_{n-2}+1$. Then, we have either $2a_{n-1} > a_{n-2}$ or $2a_n >a_{n-2}$.

Without loss of generality, we may assume that  $2a_{n-1} > a_{n-2}$. Hence we have 

$s_{n-1}w(\varpi_r)= \sum_{i \neq {n-1}}a_i\alpha_i + (a_{n-2}-a_{n-1})\alpha_{n-1} \leq w(\varpi_r)$, since $a_{n-2}-a_{n-1} < a_{n-1}$. 

On the other hand, $s_{n-1}w(\varpi_r) \geq 0$, since $a_{n-2} = a \geq a_{n-1}$. So, $s_{n-1}w > w$, a contradiction to the maximality  of $w$. 


Thus, we conclude that $a \in \{1,\frac {3}{2}\}$.

Now, if $a= \frac {3}{2}$, then clearly $r$ is odd and $G$ is not of type $B_n$. We now prove that $G$ can not be of type $C_n$. 

Suppose on the contrary, let $t$ be the least positive integer such that $\sum_{i=t}^{n-1}\alpha_i+\frac {3}{2}\alpha_n \leq w(\varpi_r)$.

Since $\langle w(\varpi_r), \check{\alpha_n} \rangle = 3-a_{n-1} \leq 2$, we have $a_{n-1} =1$.

If $t \leq n-2$, then $0\leq s_tw(\varpi_r) = \sum_{i \neq t}a_i\alpha_i < w(\varpi)$. So, $s_tw > w$, a contradiction to the maximality of $w$. Hence, $a_{n-2}=0$.

We now claim that $a_i=0 \,\, \forall \,\, i \leq n-3$. For otherwise, let $m \leq n-3$ be the largest integer such that $a_m =1$.

Now, $\langle w(\varpi_r), \check{\alpha_{m+1}+\alpha_{m+2}+\ldots \alpha_{n-1}} \rangle = -3$, a 
contradiction to the fact that $|\langle w(\varpi_r), \check{\beta} \rangle| \leq 2$ for all root 
$\beta$.

Thus, $a_i=0 \,\, \forall \,\, i \leq n-2$. Hence, $w(\varpi)= \alpha_{n-1}+\frac {3}{2}\alpha_n$.
 But $\langle w(\varpi_r), \check{\alpha_{n-1}+\alpha_n} \rangle = 3$, a contradiction to the fact that $|\langle w(\varpi_r), \check{\beta} \rangle| \leq 2$ for all root $\beta$. 

Thus, $G$ can not be of type $C_n$.

If $G$ is of type $D_n$, then $a_i \leq \frac {3}{2} \,\, \forall \,\, i = 1,2, \cdots, n$. We now claim that, $a_{n-1}+a_n \leq 2$. Suppose on the contrary, let $a_{n-1}=a_n= \frac {3}{2}$. 

We claim that $a_m = 0 \,\, \forall \,\, m \leq n-3$. Otherwise, let $t$ be the least positive integer such that $\sum_{i=t}^{n-2}\alpha_i+\frac {3}{2}\alpha_{n-1}+\frac {3}{2}\alpha_n\leq w(\varpi_r)$. Then, $a_{t-1}=0$ and $t \leq n-3$.

Hence, $\langle w(\varpi_r), \check{\alpha_t+\alpha_{t+1}+\ldots \alpha_{n-1}+\alpha_n} \rangle = 3$, a  contradiction to the fact that $|\langle w(\varpi_r), \check{\beta} \rangle| \leq 2$ for all root $\beta$. 

Thus, $a_m = 0 \,\, \forall \,\, m \leq n-3$. Hence, $w(\varpi)= \alpha_{n-2}+\frac {3}{2}(\alpha_{n-1}+\alpha_n)$.

So, $\langle w(\varpi_r), \check{\alpha_{n-2}+\alpha_{n-1}+\alpha_n} \rangle = 3$, a  contradiction to the fact that $|\langle w(\varpi_r), \check{\beta} \rangle| \leq 2$ for all root $\beta$. 

Thus, in type $D_n$ not both $a_{n-1}$ and $a_n$ can be $\frac {3}{2}$.

\end{proof}

\underline{Notation:} $J_{p,q}= \{(i_1,i_2,\cdots ,i_p): i_k \in \{1,2,\cdots ,q\} \,\, \forall\,\, k \,\,\mbox{and} \,\, i_{k+1}-i_k \geq 2\}$

Now, we will describe the set of all elements $w \in W^{I_r}$ of minimal length such that $w\varpi_r \leq 0$ for types $B_n, C_n$ and $D_n$.

\begin{theorem}

Let $W_{min}^{I_r}=$ Minimal elements of the set of all $\tau \in W^{I_r}$ such that $X(\tau)_T^{ss}(\mathcal L_{\varpi_r}) \neq \emptyset$.
 
(1) \underline{Type $B_n$}: (i) Let $r=1$. Then $w= s_ns_{n-1}\ldots s_1$.

(ii)  Let $r$ be an even integer in $\{2,3, \cdots, n\}$. For any $\underline {i}=(i_1,i_2,\cdots ,i_{\frac {r}{2}}) \in J_{\frac {r}{2},n-1}$
, there exists unique $w_{\underline {i}} \in W_{min}^{I_r}$ such that $w_{\underline {i}}(\varpi_r) = -(\sum_{k=1}^{\frac {r}{2}}\alpha_{i_k})$. Further, $W_{min}^{I_r}= \{ w_{\underline {i}}: \underline {i} \in J_{\frac {r}{2},n-1}\}$. 

(iii) Let $r$ be an odd integer in $\{2,3, \cdots, n\}$. For any $\underline {i}=(i_1,i_2,\cdots ,i_{\frac {r-1}{2}}) \in J_{\frac {r-1}{2},n-2}$
, there exists unique $w_{\underline {i}} \in W_{min}^{I_r}$ such that $w_{\underline {i}}(\varpi_r) = -(\sum_{k=1}^{\frac {r-1}{2}}\alpha_{i_k}+\alpha_n)$. Further, $W_{min}^{I_r}= \{ w_{\underline {i}}: \underline {i} \in J_{\frac {r-1}{2},n-2}\}$.

 (iv) Let $r=n$. If $n$ is even, then, $w = w_{\frac {n}{2}} \cdots w_1$, where, $w_i = s_{2i-1}\ldots s_n, \,\,  i = 1,2, \cdots \frac {n}{2}$ and if $n$ is odd, then,  $w= w_{[\frac {n}{2}]+1} \cdots w_1$, where, $w_i =  s_{2i-1}\ldots s_n, \,\,  i = 1,2, \cdots [ \frac {n}{2}]+1.$

(2) \underline{Type $C_n$}: (i) Let $r=1$. Then $w= s_ns_{n-1}\ldots s_1$.

(ii)  Let $r$ be an even integer in $\{2,3, \cdots, n-1\}$. For any $\underline {i}=(i_1,i_2,\cdots ,i_{\frac {r}{2}}) \in J_{\frac {r}{2},n-1}$
, there exists unique $w_{\underline {i}} \in W_{min}^{I_r}$ such that $w_{\underline {i}}(\varpi_r) = -(\sum_{k=1}^{\frac {r}{2}}\alpha_{i_k})$. Further, $W_{min}^{I_r}= \{ w_{\underline {i}}: \underline {i} \in J_{\frac {r}{2},n-1}\}$.  

(iii) Let $r$ be an odd integer in $\{2,3, \cdots, n-1\}$. For any $\underline {i}= (i_1,i_2,\cdots ,i_{\frac {r-1}{2}}) 
\in J_{\frac {r-1}{2},n-2}$, there exists unique $w_{\underline {i}} \in W_{min}^{I_r}$ such that 
$w_{\underline {i}}(\varpi_r) = -(\sum_{k=1}^{\frac {r-1}{2}}\alpha_{i_k}+\frac {1}{2}\alpha_n)$. Further, $W_{min}^{I_r}= \{ w_{\underline {i}}: \underline {i} \in J_{\frac {r-1}{2},n-2}\}$.

(3) \underline{Type $D_n$}: (i) Let $r=1$. Then $w= s_ns_{n-1}\ldots s_1$.

(ii) Let $r$ be an even integer in $\{2,3, \cdots, n-2\}$. For any $\underline {i}=(i_1,i_2,\cdots ,i_{\frac {r}{2}}) \in J_{\frac {r}{2},n}\setminus Z$
, there exists unique $w_{\underline {i}} \in W_{min}^{I_r}$ such that $w_{\underline {i}}(\varpi_r) = -(\sum_{k=1}^{\frac {r}{2}}\alpha_{i_k})$, where $Z = \{(i_1,i_2,\cdots ,i_{\frac {r}{2}-2},n-2,n): i_k \in \{1,2,\cdots ,n-4\} \,\,\mbox{and} \,\, i_{k+1}-i_k \geq 2 \,\, \forall \,\, k\}$. Further, $W_{min}^{I_r}= \{ w_{\underline {i}}: \underline {i} \in J_{\frac {r}{2},n}\setminus Z\}$.
 
(iii) Let $r$ be an odd integer in $\{2,3, \cdots, n-2\}$. For any $\underline {i}=(i_1,i_2,\cdots ,i_{\frac {r-1}{2}}) \in 
J_{\frac {r-1}{2},n-3}$, there exists unique $w_{\underline {i}} \in W_{min}^{I_r}$  such that $w_{\underline {i}}(\varpi_r) 
= -(\sum_{k=1}^{\frac {r-1}{2}}\alpha_{i_k}+\frac{1}{2}\alpha_{n-1}+\frac {1}{2}\alpha_n)$. Also, for any $\underline {i}=(i_1,i_2,\cdots ,i_{\frac {r-1}{2}}) \in J_{\frac {r-1}{2},n-2}$, there exists unique $w_{\underline {i},1} \in W_{min}^{I_r}$  such that $w_{\underline {i},1}(\varpi_r) 
= -(\sum_{k=1}^{\frac {r-1}{2}}\alpha_{i_k}+\frac{1}{2}\alpha_{n-1}+\frac {3}{2}\alpha_n)$ and there exists unique $w_{\underline {i},2} \in W_{min}^{I_r}$  such that $w_{\underline {i},2}(\varpi_r) 
= -(\sum_{k=1}^{\frac {r-1}{2}}\alpha_{i_k}+\frac{3}{2}\alpha_{n-1}+\frac {1}{2}\alpha_n)$. Further, $W_{min}^{I_r}= \{ w_{\underline {i}}: \underline {i} \in J_{\frac {r-1}{2},n-3}\} \bigcup  \{w_{\underline{i},j}:\underline {i} \in J_{\frac {r-1}{2},n-2}\,\, \mbox{and}\,\, j = 1,2\}$.

(iv) Let $r=n-1$ or $n$. Then, $w= \prod_{i=1}^{[\frac {n-1}{2}]}w_i$, where, $$w_i =  \left \{ \begin{array}{l} \tau_is_n ~ ~  ~ ~ \mbox {if}\,\, i \,\,\mbox {is odd}.\\ 

 \tau_is_{n-1} ~ ~   ~ ~ \mbox {if}\,\, i \,\,\mbox {is even}.\\ 

\end{array} \right .$$ 

with, $\tau_i = s_{2i-1}\ldots s_{n-2}, \,\, i = 1,2,\cdots [\frac {n-1}{2}]$.

 
\end{theorem}

\begin{proof}

{\underline {\it Proof of 1}:}

(i) $\varpi_1 = \alpha_1+\alpha_2+ \ldots +\alpha_n.$

    Take $w = s_ns_{n-1}\ldots s_1$. Then $w(\varpi_1)= -\alpha_n \leq 0$.

(ii) Let $r$ be an even integer in $\{2,3, \cdots, n-2\}$. 

We have, $\varpi_r = \sum_{i=1}^{r-1}i\alpha_i +r(\alpha_r+ \ldots +\alpha_n),\,\, 
4 \leq  r \leq (n-1).$ 

Now, $J_{\frac {r}{2},n-1}= \{(i_1,i_2,\cdots ,i_{\frac {r}{2}}): i_k \in \{1,2,\cdots ,n-1\}\,\, \mbox{and} \,\, i_{k+1}-i_k \geq 2 \,\, \forall \,\, k\}$. Consider the partial order on $J_{\frac {r}{2},n-1}$, given by $(i_1,i_2,\cdots ,i_{\frac {r}{2}}) \leq (j_1,j_2,\cdots ,j_{\frac {r}{2}})$ if $i_k \leq j_k \,\,
\forall \,\, k$ and $(i_1,i_2,\cdots ,i_{\frac {r}{2}}) < (j_1,j_2,\cdots ,j_{\frac {r}{2}})$ if 
 $i_k < j_k$ for some $k$. We will prove the theorem by induction on this order. 

For $(j_1,j_2,\cdots ,j_{\frac {r}{2}}) = (n-r+1, n-r+3, \cdots, n-1)$, we have 

$(s_{n-r+1}\ldots s_1)(s_{n-r+3}\ldots s_2)\ldots (s_{n-1}\ldots s_{\frac {r}{2}})(s_ns_{n-1}\ldots s_{\frac {r}{2}+1})(s_ns_{n-1}\ldots s_{\frac {r}{2}+2})$\\
$\ldots (s_ns_{n-1}\ldots s_r)(\varpi_r)= -(\sum_{t=1}^{\frac {r}{2}}\alpha_{n-r+2t-1})$.

Now, if $(i_1,i_2,\cdots ,i_{\frac {r}{2}}) \in J_{\frac {r}{2},n-1}$ is not maximal, then, there 
exists $t$ maximal  such that $i_t< n-r+2t-1$. 

Now, $(i_1,i_2,\cdots ,i_{t-1},1+i_t,i_{t+1},\cdots ,i_{\frac {r}{2}}) \in J_{\frac {r}{2},n-1}$ and 
$(i_1,i_2,\cdots ,i_{t-1},1+i_t,i_{t+1},\cdots ,i_{\frac {r}{2}}) > (i_1,i_2,\cdots ,i_{\frac {r}{2}})$. So, by induction, there exists $w_1 \in W^{I_r}$ such that $w_1\varpi_r = -(\sum_{k \neq t}\alpha_{i_k}+\alpha_{1+i_t})$. Taking $w= s_{1+i_t}s_{i_t}w_1$ we have $w\varpi_r = -(\sum_{k=1}^{\frac {r}{2}}\alpha_{i_k})$.     

Hence, for any $(i_1,i_2,\cdots ,i_{\frac {r}{2}}) \in J_{\frac {r}{2},n-1}$, there exists 
$w \in W^{I_r}$ of minimal length such that $w\varpi_r = -(\sum_{k=1}^{\frac {r}{2}}\alpha_{i_k})$.

Now, we will prove that the $w's$ in $W^{I_r}$ having this property are minimal. 

Let $w \in W^{I_r}$ such that $w\varpi_r = -(\sum_{k=1}^{\frac {r}{2}}\alpha_{i_k})$. 

Suppose $w$ is not minimal. Then there exist $\beta \in R^{+}$ such that 
$s_{\beta}w(\varpi_r) \leq 0$ and $l(s_{\beta}w)=l(w)-1$. Since  $s_{\beta}w(\varpi_r) \leq 0$, and  $i_{k+1}- i_k \geq 2 \,\, \forall \,\, k$, $\beta = \alpha_{i,k}$ for some $k = 1,2, \cdots \frac {r}{2}$. 

Since  $l(s_{\beta}w)=l(w)-1$, $\beta = \alpha_{i_t}$ for some $t$. Hence, $s_{\beta}w(\varpi_r)=-(\sum_{k \neq t}\alpha_{i_k}+)\alpha_{i_t} \nleq 0$, a contradiction. Thus, all the $w's$ are minimal. 

Now, it  remains to prove that for all elements of the type $-(\sum_{k=1}^{\frac {r}{2}}\alpha_{i_k})$ in the weight lattice such that $ \langle \alpha_{i_k},\alpha_{i_{k+1}} \rangle \neq 0$, for some\,\, $k$, there does not exist $w \in W^{I_r}$, of minimal length such that  $w\varpi_r = -(\sum_{k=1}^{\frac {r}{2}}\alpha_{i_k})$.

Let $\mu = -(\sum_{k=1}^{\frac {r}{2}}\alpha_{i_k})$ be such that $ \langle \alpha_{i_k},\alpha_{i_{k+1}} \rangle \neq 0$ for some $k$. Choose $k$ minimal such that $ \langle \alpha_{i_k},\alpha_{i_{k+1}} \rangle \neq 0$.

If $i_k=n-1$, then $i_{k+1}=1$ and $s_nw(\varpi_n) = -(\sum_{i_j\neq n}\alpha_{i_j})> -(\sum_{k=1}^{\frac {r}{2}}\alpha_{i_k})$. Hence, $s_nw < w$, a contradiction to the minimality of $w$.

Otherwise, $s_{i_k}w(\varpi_r)= -(\sum_{j\neq k}\alpha_{i_j})> -(\sum_{k=1}^{\frac {r}{2}}\alpha_{i_k})$. Hence, $s_{i_k}w < w$, a contradiction to the minimality of $w$.

(iii)  Let $r$ be an odd integer in $\{2,3, \cdots, n-1\}$. 

The proof is similar the case when $r$ is even.

(iv) We have, $\varpi_n = \frac {1}{2}\sum_{i=1}^{n}i\alpha_i.$

Then, $2\varpi_n = \sum_{i=1}^{n}i\alpha_i.$ 

$\underline{Case \,\, 1:}$ $n$ is even.

Take $w_i = s_{2i-1}\ldots s_n, \,\,  i = 1,2, \cdots \frac {n}{2}$.

Let $w = w_{\frac {n}{2}} \cdots w_1$. Then $w(2\varpi_n) = -\sum_{i=1}^{\frac {n}{2}}\alpha_{2i-1} \leq 0$.

$W_{min}^{I_r}=\{w_{\underline{i}}:\underline{i} \in J_{\frac {r}{2}, n-1}\}$ follows from lemma (2.1).

$\underline{\it{Case \,\, 2:}}$ $n$ is odd.

Take $w_i =  s_{2i-1}\ldots s_n, \,\,  i = 1,2, \cdots,  \frac {n+1}{2}.$

Let $w = w_{\frac {n+1}{2}} \cdots w_1$. Then $w(2\varpi_n) = -\sum_{i=1}^{\frac {n+1}{2}}\alpha_{2i-1} \leq 0$.

{\underline {\it Proof of 2}:}

(i) We have, $\varpi_1 = \alpha_1+\alpha_2+ \ldots +\frac {1}{2}\alpha_n.$

Then, $2\varpi_1 = 2(\alpha_1+\alpha_2+ \ldots +\alpha_{n-1})+\alpha_n.$

    Take $w = s_ns_{n-1}\ldots s_1$. Then $w(2\varpi_1)= -\alpha_n \leq 0$.

 Proof of (ii) and (iii) are similar to Cases (ii) and (iii) of type $B$.

 {\underline {\it Proof of 3}:} 

(i) We have, $\varpi_1 = \sum_{i=1}^{n-2}\alpha_i+\frac {1}{2}(\alpha_{n-1}+\alpha_n).$

Then, $2\varpi_1 = 2(\sum_{i=1}^{n-2}\alpha_i) +\alpha_{n-1}+\alpha_n.$

    Take $w = s_ns_{n-1}\ldots s_1$. Then $w(2\varpi_1)= -(\alpha_{n-1}+\alpha_n) \leq 0$.

Proof of (ii) and (iii) are very similar to Cases (ii) and (iii) of type $B$.

(iv) We have, $\varpi_{n-1}=\frac{1}{2}(\alpha_1+2\alpha_2+\ldots +(n-2)\alpha_{n-2})+\frac{1}{4}(n\alpha_{n-1}+(n-2)\alpha_n).$

Then, $4\varpi_{n-1}=2(\alpha_1+2\alpha_2+\ldots +(n-2)\alpha_{n-2})+n\alpha_{n-1}+(n-2)\alpha_n$

Take $$w_i =  \left \{ \begin{array}{l} \tau_is_{n-1} ~ ~  ~ ~ \mbox {if}\,\, i \,\,\mbox {is odd}.\\ 

 \tau_is_n ~ ~   ~ ~ \mbox {if}\,\, i \,\,\mbox {is even}.\\ 

\end{array} \right .$$ 

where, $\tau_i = s_{2i-1}\ldots s_{n-2}, \,\, i = 1,2,\cdots [\frac {n-1}{2}]$.

Let $w=\prod_{i=1}^{[\frac {n-1}{2}]}w_i$. Then,
 $$w(4\varpi_{n-1}) =  \left \{ \begin{array}{l} \mu-2\alpha_n  ~ ~  ~ ~ \mbox {if}\,\, n  \equiv 0 \pmod 4,\\ 

 \mu-2\alpha_{n-1} ~ ~   ~ ~ \mbox {if}\,\, n \equiv 2 \pmod 4,\\ 

 \mu-2\alpha_{n-2}-3\alpha_{n-1}-\alpha_n  ~ ~   ~ ~ \mbox {if}\,\, n  \equiv 1 \pmod 4,\\

\mu-2\alpha_{n-2}-\alpha_{n-1}-3\alpha_n  ~ ~   ~ ~ \mbox {if}\,\, n \equiv 3 \pmod 4,\\

\end{array} \right .$$ 

where, $\mu = -2(\sum_{i=1}^{[\frac {n-1}{2}]}\alpha_{2i-1})$.

 We have, $\varpi_{n}=\frac{1}{2}(\alpha_1+2\alpha_2+\ldots +(n-2)\alpha_{n-2})+\frac{1}{4}((n-2)\alpha_{n-1}+n\alpha_n).$

Then, $4\varpi_n=2(\alpha_1+2\alpha_2+\ldots +(n-2)\alpha_{n-2})+(n-2)\alpha_{n-1}+n\alpha_n$. 

Take $$w_i =  \left \{ \begin{array}{l} \tau_is_n ~ ~  ~ ~ \mbox {if}\,\, i \,\,\mbox {is odd}.\\ 

 \tau_is_{n-1} ~ ~   ~ ~ \mbox {if}\,\, i \,\,\mbox {is even}.\\ 

\end{array} \right .$$ 

where, $\tau_i = s_{2i-1}\ldots s_{n-2}, \,\, i = 1,2,\cdots [\frac {n-1}{2}]$.

Let $w= \prod_{i=1}^{[\frac {n-1}{2}]}w_i$. Then,  
$$w(4\varpi_n) =  \left \{ \begin{array}{l} \mu-2\alpha_{n-1}  ~ ~  ~ ~ \mbox {if}\,\, n \equiv 0 \pmod 4 ,\\ 

 \mu-2\alpha_n ~ ~   ~ ~ \mbox {if}\,\, n \equiv 2 \pmod 4,\\ 

 \mu-2\alpha_{n-2}-\alpha_{n-1}-3\alpha_n  ~ ~   ~ ~ \mbox {if}\,\, n \equiv 1 
\pmod 4 ,\\

\mu-2\alpha_{n-2}-3\alpha_{n-1}-\alpha_n  ~ ~   ~ ~ \mbox {if}\,\, n \equiv 3 
\pmod 4,\\

\end{array} \right .$$ 

where, $\mu = -2(\sum_{i=1}^{[\frac {n-1}{2}]}\alpha_{2i-1})$.

\end{proof}

\section{Coxeter elements admitting semistable points}
In this section, we describe all Coxeter elements $w\in W$ for which the 
corresponding Schubert variety  $X(w)$ admit a semistable point for the 
action of a maximal torus with respect to a non-trivial line bundle on $G/B$.

Now, let us assume that the root system is irreducible, see page 52 of [2]. 

\underline{\large\bf{Coxeter elements of Weyl group:}}

An elememt $w \in W$ is said to be a Coxeter element if it is of the form 
$w= s_{i_1}s_{i_2} \ldots s_{i_n}$, with $ s_{i_j} \neq s_{i_k}$ unless $j=k$, 
see page 74 of [3].

Let $ \chi = \sum_{\alpha \in S}a_{\alpha}\alpha$ be a non-zero dominant weight and 
let $w$ be a Coxeter element of $W$.

\begin{lemma}
If $w\chi \leq 0$ and $\alpha \in S$ is such that $\l(ws_{\alpha}) =
l(w)-1$, then,\\

$(1)\,\,  |\{\beta \in S \setminus\{\alpha\}: \langle \beta,\check{\alpha} 
\rangle \neq 0 \}| = 1 \,\, \mbox{or} \,\,  2.$\\

$(2)$ \,\, Further if $|\{\beta \in S \setminus\{\alpha\}: \langle \beta,
\check{\alpha} \rangle \neq 0 \}| = 2$, then $R$ must be of type $A_3$ and 
$\chi$ is of the form $a(2\alpha+\beta+\gamma)$ for some $a \in \mathbb 
Z_{\geq 0}$, where $\alpha,\beta$ and $\gamma$ are labelled as $\circ_{_\beta}$---$\circ_{_\alpha}$---$\circ_{_\gamma}$.
\end{lemma}

\begin{proof}
Since $S$ is irreducible and $\chi$ is  non zero dominant 
weight, $a_{\beta}$ is a positive rational number for each $\beta \in S$. 
Further since $w\chi \leq 0,\,\, \chi$ must be in the root lattice 
and so $a_{\beta}$ is a positive integer for every $\beta$ in $S$.

Since $w$ is a Coxeter element and $\l(ws_{\alpha}) = l(w)-1$, the 
coefficient of $\alpha$ in $w\chi$ = coefficient of $\alpha$ in 
$s_{\alpha}\chi. \hspace{8.8cm} \longrightarrow (1)$

We have \,\,  $s_{\alpha}\chi  =  \chi - \langle \chi, \check{\alpha} 
\rangle \alpha$

\hspace{2.6cm} $ =  \chi - \langle \sum_{\beta \in S}a_{\beta}\beta, 
\check{\alpha} \rangle \alpha $

 \hspace{2.6cm} $ =  \sum_{\beta \in S}a_{\beta}\beta - \sum_{\beta \in S}a_{\beta} \langle \beta, 
\check{\alpha} \rangle \alpha.$

The coefficient of $\alpha$ in $s_{\alpha}\chi$ is $-(\sum_{\beta \in S \setminus \{\alpha\}}\langle 
\beta, \check{\alpha} \rangle a_{\beta}+a_{\alpha}).\hspace{2.2cm} \longrightarrow (2) $


Since $w\chi \leq 0$, from (1) and (2) we have 

$\hspace{4cm}  -(\sum_{\beta \in S \setminus \{\alpha\}}\langle \beta, \check{\alpha} \rangle 
a_{\beta}+a_{\alpha}) \leq 0$.

Hence, \,\, $-(\sum_{\beta \in S \setminus \{\alpha\}}\langle \beta, \check{\alpha} \rangle 
a_{\beta}) \leq a_{\alpha}$ 

Thus, we have \,\, $-2(\sum_{\beta \in S \setminus \{\alpha\}}\langle \beta, \check{\alpha} 
\rangle a_{\beta}) \leq 2a_{\alpha}$. $\hspace{3.8cm}\longrightarrow (3)$

Since $\chi$ is dominant, we have,
\[
\begin{array}{rcl}
\langle \chi, \check{\beta} \rangle \geq 0, \,\, \forall \,\, \beta \in S \\
\\
\Rightarrow \langle \sum_{\gamma \in S}a_{\gamma} \gamma, \check{\beta} \rangle \geq 0\\
\\
\Rightarrow \sum_{\gamma \in S}a_{\gamma} \langle \gamma, \check{\beta} \rangle \geq 0
\end{array}
\]
Now if $\langle \beta, \check{\alpha} \rangle \neq 0$, the left hand side of the inequality is  
$ 2a_{\beta}-a_{\alpha}$ -(a non-negative integer).

Thus, we have, $2a_{\beta} \geq a_{\alpha}$ if $\langle \beta, \check{\alpha} \rangle \neq 0 
\hspace{3.8cm}\longrightarrow (4)$.

Now if $|\{\beta \in S \setminus\{\alpha\}: \langle \beta,\check{\alpha} \rangle \neq 0 \}| 
\geq 3$, from (3) and (4) we have,

\hspace{4cm} $3a_{\alpha} \leq -(2\sum_{\beta \in S \setminus \{\alpha\}}\langle \beta, 
\check{\alpha} \rangle a_{\beta}) \leq 2a_{\alpha}$.

This is a contradiction to the fact that $a_{\alpha}$ is a positive integer.

So $|\{\beta \in S \setminus\{\alpha\}: \langle \beta,\check{\alpha} \rangle \neq 0 \}| \leq 2$.

{\bf\underline{Proof of (2):}}

Suppose  $|\{\beta \in S \setminus\{\alpha\}: \langle \beta,\check{\alpha} \rangle \neq 0 \}| 
= 2$. Let $\beta, \gamma$ be the two distinct elements of this set.

Using (3) and the facts  $ \langle \beta, \check{\alpha} \rangle \leq -1$, $\langle \gamma, 
\check{\alpha} \rangle \leq -1 $,  we have 

\[
\begin{array}{rcl}
2(a_{\beta}+a_{\gamma}) \leq -2(\langle \beta, \check{\alpha} \rangle a_{\beta}+\langle \gamma, 
\check{\alpha} \rangle a_{\gamma}) \leq 2a_{\alpha}\\
\end{array}\hspace{2.5cm}\longrightarrow (5)
\]

Since $\langle \chi, \check{\beta} \rangle \geq 0$ and  $\langle \chi, \check{\gamma} \rangle 
\geq 0$ we have 

\hspace{2cm} $2a_{\beta} \geq - \sum_{\delta \neq \beta, \alpha}\langle \delta, \check{\beta} 
\rangle a_{\delta} +a_{\alpha}$ and $2a_{\gamma} \geq - \sum_{\delta \neq \gamma, \alpha}\langle 
\delta, \check{\gamma} \rangle a_{\delta} +a_{\alpha}$. 

Hence, \,\,   $- \sum_{\delta \neq \beta, \alpha}\langle \delta, \check{\beta} \rangle a_{\delta} 
- \sum_{\delta \neq \gamma, \alpha}\langle \delta, \check{\gamma} \rangle a_{\delta} +2a_{\alpha} 
\leq 2(a_{\beta}+a_{\gamma})$.

Using (5), we get 
\[
\begin{array}{lcl}
- \sum_{\delta \neq \beta, \alpha}\langle \delta, \check{\beta} \rangle  a_{\delta}- \sum_{\delta 
\neq \gamma, \alpha}\langle \delta, \check{\gamma} \rangle  a_{\delta} +2a_{\alpha} \leq 2a_{\alpha}.
\\

\\
\Rightarrow    \sum_{\delta \neq \gamma, \beta, \alpha}\langle -\delta, \check{\beta} \rangle  
a_{\delta}+ \sum_{\delta \neq \gamma, \beta, \alpha}\langle -\delta, \check{\gamma} \rangle  
a_{\delta}\leq 0,\,\, \mbox{since}\,\, \langle \beta, \check{\gamma} \rangle = \langle \gamma, \check{\beta} \rangle = 0
\end{array}
\]

Since each $a_{\delta}$ is positive and $\langle -\delta, \check{\beta} \rangle, \langle 
-\delta, \check{\gamma} \rangle$ are non-negative integers, we have 

$\langle -\delta, \check{\beta} \rangle = 0$ and $\langle -\delta, 
\check{\gamma} \rangle = 0, \,\, \forall \,\, \delta \neq \alpha, \beta,\gamma$.

Since $R$ is irreducible, we have $S= \{\alpha,\beta,\gamma \}$. 
So, from the classification theorem ( see page 57 and 58 of [2]) of irreducible root systems, 
we have 
$\langle \beta, \check{\alpha} \rangle \in \{-1,-2\}$.

If $\langle \beta, \check{\alpha} \rangle = -2 $, then $ \langle \gamma, 
\check{\alpha} \rangle = -1$.

Hence, from (3) we get 
\hspace{0.5cm}  $4a_{\beta}+2a_{\gamma} \leq 2a_{\alpha} \hspace{3cm} \longrightarrow (6)$ 

Again, from (4) we have $2a_{\beta} \geq a_{\alpha}$ and $ 2a_{\gamma} 
\geq a_{\alpha}$. So using (6), we get 
$3a_{\alpha} \leq 4a_{\beta}+2a_{\alpha}\leq 2a_{\alpha}$, a contradiction to the 
fact that $a_{\alpha}$ is a positive integer. Thus  $\langle \beta, 
\check{\alpha} \rangle = -1 $.

 Using a similar argument, we see that  $ \langle \gamma, \check{\alpha} 
\rangle = -1$.
 
Now, let us assume that  $\langle \alpha, \check{\beta} \rangle = -2$.

Then, \[
\begin{array}{rcl}
  0 \leq \langle \chi, \check{\beta} \rangle &=& a_{\gamma} \langle \gamma, 
\check{\beta} \rangle -2a_{\alpha}+2a_{\beta} \\
&=& -2a_{\alpha}+2a_{\beta}, \,\, \mbox{since} \,\,  \langle \gamma, 
\check{\beta} \rangle = 0\\
\Rightarrow 2a_{\alpha}\leq 2a_{\beta}.
\end{array}
\]
From (3), we have \hspace{1cm} $2a_{\beta}+2a_{\gamma} \leq 2a_{\alpha} \leq 
2a_{\beta}  $.\\
 Hence,  
$2a_{\gamma} \leq 0$, a contradiction.
So  $\langle \alpha, \check{\beta} \rangle = -1$. Similarly $\langle \alpha, 
\check{\gamma} \rangle = -1$.

Hence $R$ is of the type $A_3$.

\begin{center}
$\circ_{_\beta}$------$\circ_{_\alpha}$------$\circ_{_\gamma}$
\end{center}

We now show that  $\chi = a(\beta+2\alpha+\gamma)$, for some 
$ a \in \mathbb Z_{\geq 0}$. 

Let $\chi = a_{\alpha}\alpha+a_{\beta}\beta+a_{\gamma}\gamma$. By assumption, we
have $s_{\gamma}s_{\beta}s_{\alpha}(\chi) \leq 0$.

So $(a_{\beta}+a_{\gamma}-a_{\alpha})\alpha+(a_{\beta}-a_{\alpha})\gamma+
(a_{\gamma}-a_{\alpha})\beta \leq 0$.

Hence, we have $a_{\beta}+a_{\gamma} \leq a_{\alpha} \hspace{5cm} \longrightarrow (7)$  

Since $\chi$ is dominant, 
we have $\langle \chi, \check{\beta} \rangle \geq 0$ and $\langle \chi, 
\check{\gamma} \rangle \geq 0$.

So we have, $a_{\alpha}\leq 2a_{\beta}$ and $a_{\alpha}\leq 2a_{\gamma} \hspace{5cm} \longrightarrow (8)$.

Using (7) and (8),  $2a_{\alpha} \geq 2(a_{\beta}+a_{\gamma}) \geq 2a_{\alpha}$. This is 
possible only if $2a_{\beta}= a_{\alpha}= 2a_{\gamma}$.

Then, $\chi$ must be of the form $a(\beta+2\alpha+\gamma)$, for some 
$ a \in \mathbb Z_{\geq 0}$. 
\end{proof}

  Now for given an irreducible root system $R$, we describe all 
the Coxeter 
elements $w \in W$ for which there is a non-zero dominant weight $\chi$ such 
that $w\chi \leq 0$. For the Dynkin diagrams and labelling of simple roots, we refer to 
page 58 of [2].

\begin{theorem}
(A) \underline{Type $A_n$}: (1) $A_3$:  For any Coxeter element $w$, $X(w)_{T}^{ss}(\mathcal L_{\chi} \neq \emptyset$ for some non-zero dominant weight. 

(2) $A_n, n \geq 4$:  If $X(w)_{T}^{ss}(\mathcal L_{\chi} \neq \emptyset$ for some non-zero dominant weight and $w$ is a Coxeter element, then $w$ must be either $s_ns_{n-1} \ldots s_1$ or $s_i \ldots s_1s_{i+1}\ldots s_n$ for some $1 \leq i \leq n-1$.

(B) \underline{Type $B_n$}: (1) $B_2$:  For any Coxeter element $w$, $X(w)_{T}^{ss}(\mathcal L_{\chi} \neq \emptyset$ for some non-zero dominant weight.  

(2) $B_n, n \geq 3$:  If $X(w)_{T}^{ss}(\mathcal L_{\chi} \neq \emptyset$ for some non-zero dominant weight and $w$ is a Coxeter element, then $w = s_ns_{n-1} \ldots s_1$.

(C) \underline{Type $C_n$}: If $X(w)_{T}^{ss}(\mathcal L_{\chi} \neq \emptyset$ for some non-zero dominant weight and $w$ is a Coxeter element, then $w = s_ns_{n-1} \ldots s_1$.

(D) \underline{Type $D_n$}: (1) $D_4$:  $X(w)_{T}^{ss}(\mathcal L_{\chi} \neq \emptyset$ for some non-zero dominant weight and $w$ is a Coxeter element if and only if  $l(ws_2)=l(w)+1$.

(2) $D_n, n \geq 5$: If $X(w)_{T}^{ss}(\mathcal L_{\chi} \neq \emptyset$ for some non-zero dominant weight and $w$ is a Coxeter element, then $w = s_ns_{n-1} \ldots s_1$.
 
(E) \underline{ $E_6, E_7, E_8$}:  There is no Coxeter element $w$ for which there exist a 
non-zero dominant weight $\chi$ such that  $X(w)_{T}^{ss}(\mathcal L_{\chi} \neq \emptyset$.

(F) \underline{ $F_4$}:  There is no Coxeter element $w$ for which there exist 
a non-zero dominant weight $\chi$ such that  $X(w)_{T}^{ss}(\mathcal L_{\chi} \neq \emptyset$.

(G) \underline{ $G_2$}: There is no Coxeter element $w$ for which there exist 
a non-zero dominant weight $\chi$ such that  $X(w)_{T}^{ss}(\mathcal L_{\chi} \neq \emptyset$.

\end{theorem}

\begin{proof}

By lemma (2.1), $X(w)_{T}^{ss}(\mathcal L_{\chi} \neq \emptyset$ for a non-zero dominant weight $\chi$ if and only if $w\chi \leq 0$. So, using this lemma we investigate all the cases.

{\bf\underline{Proof of (A)}:}


$(1)$ The Coxeter elements of $A_3$ are precisely $s_1s_2s_3, s_1s_3s_2, s_2s_1s_3, s_3s_2s_1$. 
For $w =s_1s_3s_2$, take $\chi = \alpha_1+2\alpha_2+\alpha_3$. Otherwise take $\chi = \alpha_1
+\alpha_2+\alpha_3$. Then $w \chi \leq 0$.

$(2)$ Let $n \geq 4$, and let $w \chi \leq 0$ for some dominant weight $\chi$. By lemma (4.1), 
if $l(ws_i) = l(w)-1$, then $i=1$ or $i= n$.

If  $l(ws_n) \neq  l(w)-1$, then using the fact that $s_i$ commute with $s_j$ for $j 
\neq i-1,i+1$, it is easy to see that $w = s_ns_{n-1}\ldots s_2s_1$. 

If  $l(ws_n) =  l(w)-1$, then, let $i$ be the least integer in $\{1,2,\cdots , n-1\}$ such that 
$w = \phi s_{i+1}\ldots s_n$, for some $\phi \in W$ with $l(w)= l(\phi)+(n-i)$. Then, we have 
to show that $\phi = s_is_{i-1}\ldots s_1$.

If $\phi = \phi_1s_j$ for some $j \in \{2,3, \cdots , i-1\}$, then $w$ is of the form 
\[
\begin{array}{lcl}
w &=& \phi_1s_j(s_{i+1}\ldots s_{n-1}s_n) \\
  &=& \phi_1(s_{i+1}\ldots s_{n-1}s_ns_j). 
\end{array}
\]
This contradicts lemma (4.1). So $j \in \{1,i\}$. Again $j =i$ is not possible 
unless $i=1$ by the minimality of $i$.

Thus, we have $\phi = s_i\ldots s_1$.


{\bf\underline{Proof of (B):}}


(1) For $w = s_1s_2$,  take $\chi = \alpha_1+2 \alpha_2$. 

For $w = s_2s_1$, take $\chi = \alpha_1+\alpha_2$.

(2) For $ w = s_ns_{n-1}\ldots s_1$, take $\chi = \alpha_1+ \alpha_2+ \ldots \alpha_n.$ 
Then $w\chi = -\alpha_n \leq 0$.

Conversely, let $w$ be a Coxeter element and let $\chi$ be a non-zero dominant 
weight such that $w\chi \leq 0$. By lemma (4.1), if $l(ws_i)=l(w)-1$ then either $i=1$ or $i=n$.

If  $l(ws_n) \neq  l(w)-1$, then using the fact that $s_i$ commute with $s_j$ for $j \neq i-1,i+1$, 
it is easy to see that $w = s_ns_{n-1}\ldots s_2s_1$.

We now claim that $l(ws_n)=  l(w)+1$. If not, then, the coefficient of $\alpha_n$ in $w\chi =$ 
coefficient of $\alpha_n$ in $s_n\chi$. 

Now, the coefficient of $\alpha_n$ in $s_n\chi$ is $ 2a_{n-1}-a_n$. Since $w\chi \leq 0$, we have 
$2a_{n-1}-a_n \leq 0$.

$\hspace{4cm} \Rightarrow 2a_{n-1} \leq a_n$. $\hspace{3.8cm}\longrightarrow (1)$

Since $\chi$ is dominant, we have $\langle \chi, \check{\alpha_{n-1}} \rangle \geq 0$. Thus, 
we get

\hspace{4cm} $-a_{n-2}+2a_{n-1}-a_n \geq 0$.

\hspace{4cm} $\Rightarrow a_{n-2}\leq 2a_{n-1}-a_n \leq 0$, by (1).

So $a_{n-2}= 0$, a contradiction to the assumption that $n \geq 3$ and $\chi$ is a non-zero 
dominant weight. Thus $l(ws_n)=  l(w)+1$.

So the only possibility for $w$ is $s_ns_{n-1}\ldots s_1$.  

{\bf\underline{Proof of (C)}:}


For $  w =s_ns_{n-1}\ldots s_1$, take $\chi = 2(\sum_{i \neq n}\alpha_i)+\alpha_n$. Then, $\chi$
 is dominant and $w\chi = -\alpha_n$.

Conversely, let $w$ be a Coxeter element and let $\chi$ be a non-zero dominant weight such that 
$w\chi \leq 0$. By lemma (4.1), if $l(ws_i)=l(w)-1$ then $i \in \{1,n\}$.

If  $l(ws_n) \neq  l(w)-1$, then using the fact $s_i$ commute with $s_j$ for $j \neq i-1,i+1$, 
it is easy to see that $w = s_ns_{n-1}\ldots s_2s_1$.

{\it Claim}: $l(ws_n)=  l(w)+1$. 

If not, then, the coefficient of $\alpha_n$ in $w\chi =$ coefficient 
of $\alpha_n$ in $s_n\chi$.

Now, the coefficient of $\alpha_n$ in $s_n\chi$ is $a_{n-1}-a_n$. Since $w\chi \leq 0$, we have 
$a_{n-1}-a_n \leq 0$.

Hence, we have  $a_{n-1} \leq a_n$. $\hspace{3.8cm}\longrightarrow (2)$

Since $\chi$ is dominant, we have $\langle \chi, \check{\alpha_{n-1}} \rangle \geq 0$. Thus,
 we get

\hspace{4cm} $-a_{n-2}+2a_{n-1}-2a_n \geq 0$.

\hspace{4cm} $\Rightarrow  a_{n-2}\leq 2a_{n-1}-2a_n \leq 0$, by (2).

So $a_{n-2}= 0$, a contradiction to the assumption that $\chi$ is a non-zero dominant weight. 

Thus $l(ws_n)=  l(w)+1$.

So the only possibility for $w$ is $s_ns_{n-1}\ldots s_1$.

{\bf\underline{Proof of (D)}:}


(1) For $w = s_4s_3s_2s_1$, 
take $\chi = 2(\alpha_1+\alpha_2)+\alpha_3+\alpha_4$, for  $w = s_4s_1s_2s_3$, take $\chi = 2(\alpha_3+\alpha_2)+\alpha_1+\alpha_4$ and
for $w = s_3s_1s_2s_4$, take $\chi = 2(\alpha_4+\alpha_2)+\alpha_1+\alpha_3$. 

The converse follows from lemma (4.1).



  



(2) For $  w =s_ns_{n-1}\ldots s_1$, take $\chi = 2(\sum_{i=1}^{n-2}\alpha_i)+\alpha_{n-1}+\alpha_n$. Then $w\chi \leq 0$.

Conversely, let $w$ be a Coxeter element and let $\chi$ be a non-zero dominant 
weight such that $w\chi \leq 0$. By lemma (4.1), if $l(ws_i)=l(w)-1$ then $i \in \{1,n-1,n\}$.

Now, if  $l(ws_1) =  l(w)-1$, then, it is easy to see that $w = s_ns_{n-1}\ldots s_2s_1$.

So, it is sufficient to prove that $l(ws_n)=  l(w)+1$ and $l(ws_{n-1})=  l(w)+1$. 

If $l(ws_n)=  l(w)-1$, then, the 
coefficient of $\alpha_n$ in $w\chi =$ coefficient of $\alpha_n$ in $s_n\chi 
=  a_{n-2}-a_n$.

 Since $w\chi \leq 0$, we have $a_{n-2}-a_n \leq 0$.$\hspace{3cm}
\longrightarrow (4)$

Since $\chi$ is dominant we have $\langle \chi, \check{\alpha_{n-2}} \rangle \geq 0$. Therefore, we have 

\hspace{4cm} $2a_{n-2}\geq a_{n-1}+a_{n-3}+a_n $.$\hspace{2.8cm}\longrightarrow (5)$

Also, since  $\langle \chi, \check{\alpha_{n-1}} \rangle \geq 0$  and $\langle \chi, 
\check{\alpha_{n-3}} \rangle \geq 0$, we have

 \hspace{4cm} $2a_{n-1}- a_{n-2}\geq 0$  $\hspace{3.8cm}\longrightarrow (6)$

 and \hspace{3cm} $2a_{n-3}- a_{n-4}-a_{n-2}\geq 0 $. $\hspace{2.8cm}\longrightarrow (7)$

From (5), we get 

 \hspace{4cm} $4a_{n-2}\geq 2a_{n-1}+2a_{n-3}+2a_n $

 \hspace{5cm} $\geq a_{n-2}+(a_{n-4}+a_{n-2})+2a_n$, from (6) and (7)

  \hspace{5cm} $\geq 2a_{n-2}+2a_{n-2}+a_{n-4}$, by (4)

 \hspace{5cm} $= 4a_{n-2}+a_{n-4}$.

So $a_{n-4} = 0$, a contradiction to the assumption that $\chi$ is a non-zero dominant weight. 
So $l(ws_n)=  l(w)+1$.

Using a similar argument, we can show that $l(ws_{n-1})=  l(w)+1$. 

{\bf\underline{Proof of (E):}}




$\underline{\large\bf{Type \,\, E_8:}}$

Let $w$ be a Coxeter element  and  let $\chi$ be a non-zero dominant 
weight $\chi$ such that $w\chi \leq 0$. Further, if $l(ws_i)=l(w)-1$, then by 
lemma (4.1),  $i \in \{1,2,8\}$.

$\it{Case \,\,  1: i = 8}$

Co-efficient of $\alpha_8$ in $w\chi =$ Co-efficient of $\alpha_8$ in $s_8(\chi)=
a_7-a_8 \leq 0$.

Since $\chi$ is dominant, $\langle \chi, \check{\alpha_i} \rangle \geq 0\,\, \forall \,\, i \in 
\{1,2,3,4,5,6,7,8\}$. 
 
$\langle \chi, \check{\alpha_7} \rangle \geq 0 \Rightarrow 2a_7 \geq a_6+a_8 
\geq a_6+a_7$.

Hence, we have  $a_7 \geq a_6$.

$\langle \chi, \check{\alpha_6} \rangle \geq 0 \Rightarrow 2a_6 \geq a_5+a_7 
\geq a_5+a_6$

 \hspace{5cm} $\Rightarrow a_6 \geq a_5$.

$\langle \chi, \check{\alpha_5} \rangle \geq 0 \Rightarrow 2a_5 \geq a_4+a_6 
\geq a_4+a_5$.

\hspace{5cm} $\Rightarrow a_5 \geq a_4$

$\langle \chi, \check{\alpha_3} \rangle \geq 0 \Rightarrow 2a_3 \geq a_1+a_4$.

$\langle \chi, \check{\alpha_2} \rangle \geq 0 \Rightarrow 2a_2 \geq a_4$.

Now, $\langle \chi, \check{\alpha_4} \rangle \geq 0 \Rightarrow 2a_4 \geq a_2
+a_3+a_5$

\hspace{3cm} $\Rightarrow 4a_4 \geq 2a_2+2a_3+2a_5$.

\hspace{4cm} $\geq a_4+a_1+a_4+2a_4$, since $a_5 \geq a_4$.
  
So, $a_1=0$.
Thus in this case, there is no Coxeter element $w$ for which there is a non-zero 
dominant weight such that $w\chi \leq 0$.

$\it{Case \,\, 2: i = 1}$

Co-efficient of $\alpha_1$ in $w\chi =$ Co-efficient of $\alpha_1$ in $s_1\chi=
a_3-a_1 \leq 0$.

Since $\chi$ is dominant, we have $\langle \chi, \check{\alpha_3} \rangle \geq 0$. Therefore,
$ 2a_3 \geq a_1+a_4 \geq a_3+a_4$ 

Hence, we have  $ a_3 \geq a_4$ .

Since,  $\langle \chi, \check{\alpha_4} \rangle \geq 0$, we have $2a_4 \geq a_3+a_2+a_5$.

Since, $\langle \chi, \check{\alpha_2} \rangle \geq 0$ and $\langle \chi, \check{\alpha_5} 
\rangle \geq 0$ we have $2a_2 \geq a_4$ and $2a_5 \geq a_4+a_6$.

Then, $4a_4 \geq 2a_3+2a_2+2a_5 \geq 2a_4+a_4+a_4+a_6,$ from the above inequalities.

So, $a_6 = 0$. Hence we have $\chi = 0$.  Thus, in this case also, there no Coxeter element $w$ 
for which there exist a non-zero dominant weight $\chi$ such that $w\chi \leq 0$.

$\it{Case \,\, 3: i = 2}$

Co-efficient of $\alpha_2$ in $w\chi =$ Co-efficient of $\alpha_2$ in $s_2\chi=
a_4-a_2 \leq 0$.

Since $\chi$ is dominant, $\langle \chi, \check{\alpha_i} \rangle \geq 0\,\, \forall \,\, i \in 
\{1,2,3,4,5,6\}$.

$\langle \chi, \check{\alpha_5} \rangle \geq 0 \Rightarrow 2a_5 \geq a_4+a_6 $.

$\langle \chi, \check{\alpha_3} \rangle \geq 0 \Rightarrow 2a_3 \geq a_1+ a_4$.

$\langle \chi, \check{\alpha_4} \rangle \geq 0 \Rightarrow 2a_4 \geq a_3+a_2+a_5$.

 Hence, we have $ 4a_4 \geq 2a_3+2a_2+2a_5$.

\hspace{3.2cm} $ \geq (a_1+a_4)+2a_4+(a_4+a_6) = a_1+a_6+4a_4$.

$\Rightarrow  a_1+a_6 = 0$. So, $a_1=a_6 = 0$. 

Hence, we have  $\chi=0$. Thus, in this case also, there no Coxeter element $w$ 
for which there exist a non-zero dominant weight $\chi$ such that $w\chi \leq 0$.

$\underline{\large\bf{Type \,\, E_6, E_7:}}$

Proof is similar to the case of $E_8$.

{\bf\underline{Proof of F}:}


Let $w$ be a Coxeter element. Let $\chi$ be a non-zero dominant weight such that $w\chi 
\leq 0$. If $l(ws_i)=l(w)-1$, then $i \in \{1,4\}$, by lemma (4.1). 

$\it{Case \,\, 1:i = 1}$

Co-efficient of $\alpha_1$ in $w\chi =$ Co-efficient of $\alpha_1$ in $s_1\chi=
a_2-a_1 \leq 0$.

Since $\chi$ is dominant, we have $\langle \chi, \check{\alpha_3} \rangle \geq 
0$ and $\langle \chi, \check{\alpha_2} \rangle \geq 0$.

$\langle \chi, \check{\alpha_2} \rangle \geq 0 \Rightarrow 2a_2 \geq a_1+a_3 \geq a_2+a_3$, since $a_2 \leq a_1$.

Hence, we have  $a_2 \geq a_3$.

$\langle \chi, \check{\alpha_3} \rangle \geq 0 \Rightarrow 2a_3 \geq 2a_2+a_4 \geq 2a_3+a_4$.

So, we have $ a_4 = 0$. Hence, $\chi = 0$. Thus, in this case there no Coxeter element $w$ 
for which there exist a non-zero dominant weight $\chi$ such that $w\chi \leq 0$.

$\it{Case \,\, 2: i = 4}$

Co-efficient of $\alpha_4$ in $w\chi =$ Co-efficient of $\alpha_4$ in $s_4\chi=
a_3-a_4 \leq 0$.

Since $\chi$ is dominant, we have $\langle \chi, \check{\alpha_3} \rangle \geq 0$ and $\langle 
\chi, \check{\alpha_2} \rangle \geq 0$.

$\langle \chi, \check{\alpha_3} \rangle \geq 0 \Rightarrow 2a_3 \geq 2a_2+a_4 \geq 2a_2+a_3$, since $a_3 \leq a_4$.

Hence, we have  $a_3 \geq 2a_2$. 

$\langle \chi, \check{\alpha_2} \rangle \geq 0 \Rightarrow 2a_2 \geq a_1+a_3 \geq a_1+2a_2$.

So, we have  $ a_1=0$. Hence, $\chi = 0$. Thus, in this case also, there no Coxeter element $w$ 
for which there exist a non-zero dominant weight $\chi$ such that $w\chi \leq 0$.

{\bf\underline{Proof of G}:}


Let $w$ be a Coxeter element and  $\chi = a_1\alpha_1+a_2\alpha_2$, be a 
dominant weight such that $w\chi \leq 0$

$\it{Case \,\,1: l(ws_1)=l(w)-1}$.

Co-efficient of $\alpha_1$ in $w\chi =$ Co-efficient of $\alpha_1$ in $s_1\chi=
a_2-a_1 \leq 0$.

Since $\chi$ is dominant, we have $\langle \chi, \check{\alpha_2} \rangle \geq 0$.

\hspace{5cm} $\Rightarrow 2a_2 \geq 3a_1 \geq 3a_2$.

So, we have  $ a_2=0$. Hence, $\chi = 0$. Thus, in this case, there no Coxeter element $w$ 
for which there exist a non-zero dominant weight $\chi$ such that $w\chi \leq 0$.

$\it{Case \,\, 2: l(ws_2)=l(w)-1}$.

Co-efficient of $\alpha_2$ in $w\chi =$ Co-efficient of $\alpha_2$ in $s_2\chi=
3a_1-a_2 \leq 0$.

Since $\chi$ is dominant, we have $\langle \chi, \check{\alpha_1} \rangle \geq 0$.

\hspace{5cm} $\Rightarrow 2a_1 \geq a_2 \geq 3a_1$.

So, we have $ a_1=0$. Hence, $\chi = 0$. Thus, in this case also, there no Coxeter element $w$ 
for which there exist a non-zero dominant weight $\chi$ such that $w\chi \leq 0$.

\end{proof}


\begin{thebibliography}{22}

 \bibitem[1]{r1} R.W. Carter, Finite Groups of Lie type, John Wiley, New York, 
1993.
\bibitem[2]{r2} J.E. Humphreys, Introduction to Lie algebras and representation theory, Springer, Berlin Heidelberg, 1972.
\bibitem[3]{r3} J.E. Humphreys, Reflection Groups and Coxeter Groups, Cambridge Univ. Press, Cambridge, 1990.
\bibitem[4]{r4} S. S. Kannan, Torus quotients of homogeneous spaces, Proc. Indian Acad. Sci.(Math. Sci), 108(1998), no 1, 1-12.
\bibitem[5]{r5} S. S. Kannan, Torus quotients of homogeneous spaces-II, Proc. Indian Acad. Sci.(Math. Sci), 109(1999), no 1, 23-39.
\bibitem[6]{r6} S.S. Kannan, Cohomology of line bundles on Schubert varieties in the Kac-Moody setting, J. Algebra, 310(2007) 88--107.
\bibitem[7]{r7}  S. S. Kannan, Pranab Sardar, Torus quotients of homogeneous 
spaces of the general linear group and the standard representation of certain 
symmetric groups, to appear in Proc. Indian Acad. Sci.
\bibitem[8]{r8} D. Mumford, J. Fogarty and F. Kirwan, Geometric Invariant theory,
(Third Edition), Springer-Verlag, Berlin Heidelberg, New York, 1994. 
\bibitem[9]{r9} P.E. Newstead, Introduction to Moduli Problems and Orbit 
Spaces, TIFR Lecture Notes, 1978.
\bibitem[10]{r10} C.S. Seshadri, Quotient spaces modulo reductive algebraic 
groups, Ann. Math. 95(1972) 511-556.   
\bibitem[11]{r11} C.S. Seshadri, Introduction to Standard Monomial theory, 
Lecture notes No.4, Brandeis University, Waltham, MA, 1985.


\end{thebibliography}
\end{document}